\documentclass[a4paper]{ltxdoc}
\usepackage{hwstyle}
\usepackage{authblk}

\usepackage[norule]{footmisc}

\title{A Closer Look at the Multilinear Cryptography\\ using Nilpotent Groups}

\author{Delaram Kahrobaei, Antonio Tortora, Maria Tota}

\date{ }

\begin{document}
\maketitle

\begin{abstract}
In a previous paper we generalized the definition of a multilinear map to arbitrary groups and introduced two multiparty key-exchange protocols using nilpotent groups. In this paper we have a closer look at the protocols and will address some incorrect cryptanalysis which have been proposed.\\
 

\end{abstract}

\section{Introduction}

Let $n$ be a positive integer. In \cite{KTT} we generalized the definition of a multilinear map for cyclic groups of the same prime order (see \cite{BS}) to arbitrary groups $G_1,\dots,G_n$ and $G_T$. Indeed, we defined a map 
$$e:G_1\times \dots\times G_n \rightarrow G_T$$ 
to be an $n$-linear map (or a multilinear map) if for any $a_1, \dots, a_n \in \mathbb{Z}$ and any $g_i\in G_i$, we have $$e(g_1^{a_1}, \dots, g_n^{a_n}) = e(g_1, \dots, g_n)^{a_1 \cdots a_n}.$$
The map $e$ is symmetric when $G_1=\dots=G_n=G$, and
non-degenerate if there exists $g\in G$ such that $e(g,\stackrel{n}{\dots},g)\neq 1$. Notice also that $e$ is not necessarily  linear in each component. 


In this paper we have a closer look at two multiparty key-exchange protocols introduced in \cite{KTT}.  Also, we will address some incorrect cryptanalysis which have been proposed in \cite{Rom}.

The protocols are based on the use of nilpotent group identities. Recall that a group $G$ is nilpotent if it has a finite series of subgroups
$$G=H_{0}>H_1>\dots>H_{n}=\{1\}$$
which is central, that is, each $H_i$ is normal in $G$ and $H_{i}/H_{i+1}$ is contained in the center of $G/H_{i+1}$. The length of a shortest central series is the nilpotency class of $G$. Hence, nilpotent groups
of class at most 1 are abelian. A great source of nilpotent groups is the class of finite $p$-groups, i.e., finite groups whose orders are powers of a prime $p$.

\section{The Protocols}

Let $G$ be a nilpotent group of class $n>1$ and let $g_{1},\dots, g_{n+1}$ be elements of $G$. Then, by a characterization of nilpotent groups, we have
$$[g_1,\dots,g_{n+1}]=1$$
where the commutator $[g_1,\dots,g_{n+1}]$ is defined recursively by the rules
$$[g_1, g_2] = g_1^{-1} g_2^{-1} g_1 g_2 \quad 
{\rm and} \quad  
[g_{1},\dots,g_{n+1}]=
[[g_{1},\ldots,g_{n}],g_{n+1}].$$ 
According to Proposition 3 of \cite{KTT}, for any $i\in\{1,\dots,n\}$ and $a_i\in \mathbb{Z}\backslash\{0\}$, it follows that
$$[g_1,\dots, g_{i-1}, g_i^{a_i}, g_{i+1},\dots,g_{n}]=[g_1,\dots, g_{i-1}, g_i, g_{i+1},\dots,g_{n}]^{a_i}.$$
This allows to construct the multilinear map $e:G^{n} \rightarrow G$ given by
$$e({g_1},\dots, {g_{n}})=[g_1,\dots,g_{n}].$$
Notice that, for $n=2$, $e$ is the bilinear map which has been mentioned in \cite{MS}. 

A group $G$ is said to be $n$-Engel, with $n\geq 1$, if 
$$[x,_n g]=[x, \underbrace{g,\dots, g}_n]=1$$ 
for all $x,g \in G$. Of course any nilpotent group of class $n$ is $n$-Engel. Also, it is well-known that there exist 
nilpotent groups of class $n+1$ which are not $n$-Engel (see, for instance, \cite[Theorem 6.2]{Liebeck}). 
Thus, if $G$ is such a group and $x,g$ are elements of $G$ such that $[x,_n g]\neq 1$, one can consider the non-degenerate multilinear map $e':G^n \rightarrow G$ given by
$$e'({g_1},\dots, {g_{n}})=[x,{g_1},\dots, {g_{n}}].$$

In \cite{KTT} we proposed the following two key exchange protocols based on the multilinear maps $e$ and $e'$, respectively. In both cases, we have $n+1$ users ${\cal{A}}_1,\dots,{\cal{A}}_{n+1}$ with private nonzero integers $a_1, \dots, a_{n+1}$, respectively, who want to agree on a shared secret key. 

\bigskip

\noindent {\bf Protocol I}

\medskip

Let $G$ be a public nilpotent group of class $n>1$ and let $g_1, \dots, g_n \in G$ such that $[g_1,\dots,g_n]\neq 1$.

\begin{itemize}
\item The users ${\cal{A}}_1$ and ${\cal{A}}_{n+1}$ transmit in public channel ${g_1}^{a_1}$ and ${g_n}^{a_{n+1}}$, respectively; the user ${\cal{A}}_j$ ($j =2, \dots, n$) transmits ${g_{j-1}}^{a_j}$ and ${g_j}^{a_j}$.
\item The user ${\cal{A}}_1$ computes $[{g_1}^{a_2}, \dots, {g_{n}}^{a_{n+1}}]^{a_1}$.
\item The user ${\cal{A}}_j$ ($j =2, \dots, n$) computes $$[{g_1}^{a_1}, \dots, {g_{j-1}}^{a_{j-1}}, {g_j}^{a_{j+1}}, {g_{j+1}}^{a_{j+2}}, \dots, {g_{n}}^{a_{n+1}}]^{a_j}.$$

\item The user ${\cal{A}}_{n+1}$ computes $[{g_1}^{a_1}, \dots, {g_{n}}^{a_{n}}]^{a_{n+1}}$.
\end{itemize}
Hence, each user obtains $[g_1,\dots, g_{n}]^{\prod_{j=1} ^{n+1} a_j}$ which is the shared key.

\bigskip

\noindent {\bf Protocol II}

\medskip

Let $G$ be a public nilpotent group of class $n+1$ which is not $n$-Engel ($n\geq 1$) and let $x,g\in G$ such that $[x,_n g]\neq 1$. 
\begin{itemize}
\item Each user ${\cal{A}}_j$ computes $g^{a_j}$ and sends it to the other users.
\item The user ${\cal{A}}_1$ computes $[x^{a_1},g^{a_2},\dots,g^{a_{n+1}}]$.
\item The user ${\cal{A}}_j$ $(j=2,\dots,n)$ computes $[x^{a_j}, g^{a_1}, \dots,g^{a_{j-1}},g^{a_{j+1}},\dots, g^{a_{n+1}}]$.
\item The user ${\cal{A}}_{n+1}$ computes $[x^{a_{n+1}},g^{a_1},\dots,g^{a_n}]$.
\end{itemize}
The common key is $[x,_n g]^{{\prod}_{j=1} ^{n+1} a_j}$.

\subsection{Platform groups}

As a basis for the key exchange methods described above, we suggest to consider finitely generated nilpotent groups. These groups are polycyclic, and in particular supersoluble (see, for instance, \cite[5.4.6]{Rob}).
Recall that a group $G$ is said to be polycyclic if it has a finite cyclic series, that is, a sequence of subgroups
\begin{equation}\label{poly}
G=G_{1}>G_2>\dots>G_{n+1}=\{1\} 
\end{equation}
such that, for $1\leq i\leq n$, $G_{i+1}$ is normal in $G_{i}$ and $G_{i}/G_{i+1}$ is cyclic. The group $G$ is then supersoluble when the series is normal, i.e., each $G_{i+1}$ is normal in $G$.

Let $G$ be a polycyclic group with the cyclic series $(\ref{poly})$. Following \cite{HEO}, we say that a sequence $X=(x_1,\dots, x_n)$ of elements of $G$ such that $G_{i}/G_{i+1}=\langle x_{i}G_{i+1} \rangle$, for $1\leq i\leq n$, is a polycyclic sequence for $G$; hence, each  $G_i$ is generated by $x_{i},\dots, x_n$. Defining $r_i=\vert G_i:G_{i+1}\vert\in \mathbb{N}\cup \{\infty\}$, the sequence
$R(X)=(r_1,\dots, r_n)$ is called the sequence
of relative orders for $X$. The set $\{i\in\{1,\dots, n\}\,\vert\, r_i\,{\rm is\,finite}\}$ is usually denoted by $I(X)$.
If $X$ is a polycyclic sequence for $G$ with the relative orders $R(X)=(r_1,\dots,r_n)$, then for any $g\in G$ there exists a unique sequence $(e_1,\dots,e_n)$ of integers, with $0\leq e_i< r_i$ if $i\in I(X)$, such that
$$g=x_1^{e_1}\dots x_n^{e_n}$$
(see \cite[Lemma 8.3]{HEO}). This latter expression is the normal form of $g$ with respect to $X$, and $(e_1,\dots,e_n)$ is the exponent vector of $g$ with respect to $X$. 

A group presentation $\langle x_1,\dots,x_n\,\vert\, R\rangle$ is called a nilpotent presentation if there exists a sequence  $S=(s_1,\dots,s_n)$, with $s_i\in\mathbb{N}\cup\{\infty\}$, and integers $a_{i,k},b_{i,j,k},c_{i,j,k}$ such that $R$ consists of the following relations:
\begin{align*}
x_i^{s_i} & = x_{i+1}^{a_{i,i+1}}\dots x_n^{a_{i,n}}\quad {\rm for}\; 1\leq i\leq n\; {\rm with}\; s_i\in\mathbb{N}, \\
x_j^{-1}x_i x_j & = x_i x_{i+1}^{b_{i,j,i+1}}\dots x_n^{b_{i,j,n}}\quad {\rm for}\; 1\leq j<i\leq n,\\
x_j x_i x_j^{-1} & = x_i x_{i+1}^{c_{i,j,i+1}}\dots x_n^{c_{i,j,n}} \quad {\rm for}\; 1\leq j<i\leq n.
\end{align*}
Every finitely generated nilpotent group has a polycyclic sequence $X$ which induces a consistent nilpotent presentation, where consistent means that $R(X)=S$; conversely, every consistent nilpotent presentation defines a finitely generated nilpotent group (see \cite[Lemma 8.23]{HEO}). Furthermore, for a group $G$ given by a consistent nilpotent presentation $\langle x_1,\dots,x_n\,\vert\, R\rangle$, there exists a method that allows to determine the normal form of any $g\in G$ with respect to $X=(x_1,\dots,x_n)$  (see \cite[Subsection 8.1.3]{HEO}). This is the so-called collection algorithm, which is implemented in GAP \cite{GAP} and MAGMA \cite{MAGMA}, and it has proved to be practical for finite and infinite groups.

In the context of Protocol II (and similarly for Protocol I), assuming that $\langle x_1,\dots,x_n\,\vert\, R\rangle$ is a consistent nilpotent presentation of $G$ and that $x_1^{e_{1}}\dots x_n^{e_{n}}$ is the normal form of $g$, we can make public the exponent vector $(e_{1},\dots,e_{n})$ and require the following: each user ${\cal{A}}_j$ computes the normal form $x_1^{a_{j1}}\dots x_n^{a_{jn}}$ of $g^{a_j}$ and sends the exponent vector $(a_{j1},\dots,a_{jn})$ to the other users.


\section{Cryptanalysis and Attacks}

The security of our protocols is based on the Power Search Problem (PSP): given a group $G$ and elements $g,h\in G$, find an integer $a$ such that $g^a=h$. This is actually equivalent to the Discrete Logarithm Problem (DLP) in the cyclic group generated by $g$.

Recently, for a finite nilpotent group, an algorithm to solve the PSP has been introduced in \cite{Rom}. However, we point out that it is not practical. To this end, assume that $G$ is a finite $p$-group. Also, let $G=G_{0}$ and for $i\geq 0$ define recursively $G_{i+1}=G_i^pG_i'$, where $G_i^p=\langle g_i^p\,\vert\,g_i\in G_i\rangle$ and $G_i'=\langle [g_{i1},g_{i2}]\,\vert\,g_{i1},g_{i2}\in G_i\rangle$. Since the order of $G$ is a power of $p$, one can consider in $G$ the normal series 
\begin{equation}\label{series}
G=G_{0}>G_1>\dots>G_{n}=\{1\}.
\end{equation} 
Notice that each factor $G_i/G_{i+1}$ is an abelian group of exponent $p$. Therefore $g^{p^{k}}\in G_k$ for any $g\in G$ and any $k\in\{1,\dots,n\}$. 

Now suppose $g^a=h$, for some $g, h\in G$, and write $a=a_0+a_1p+\dots+a_{n-1}p^{n-1}$ where $0\leq a_i<p$. Then 
\begin{equation}\label{DLP}
(gG_1)^{a_0}=hG_1.
\end{equation}
 The first step of the algorithm in \cite{Rom} consists in finding $a_0$. In particular, when $g,h\not\in G_1$, the author affirms that the ``exponent $a_0$ is uniquely computed by usual computation with vectors in $G/G_1$''. After examining the case $(g\not\in G_1,h\in G_1)$, the process continues until all the integers $a_i$ are obtained.

Actually it could be very difficult to find $a_0$ when the prime $p$ is big enough. For example, let $p$ be a safe prime, i.e, $p=2q+1$ with $q$ prime. Assume further that $G=\langle g^2\rangle$, where $g$ is a generator of the multiplicative group of integers modulo $p$. Thus the series (\ref{series}) becomes $G=G_{0}>G_1=\{1\}$. Hence, finding $a_0$ from (\ref{DLP}) essentially means to solve the DLP of $h$ with respect to $g$. However $G$ has order $q$ and, when $q$ is very large, it is well-known that the DLP is hard in $G$ (without the use of a quantum algorithm \cite{Shor}).

\subsection{The PSP for Unitriangular Matrices}

Let $R$ be a commutative ring with identity $1$ and denote by $UT(n, R)$,  where $n>1$, the group of all $n\times n$ (upper) unitriangular matrices over $R$, that is (upper) triangular matrices with $1$ on the diagonal. Then $UT(n, R)$ is a nilpotent group of class $n-1$, which is generated by finitely many elements when $R=\mathbb{Z}_m$ or $\mathbb{Z}$ (see, for instance, \cite[Section 5.4]{Rob}). Moreover every finite $p$-group can be embedded in $UT(n, F_p)$, for a finite field $F_p$ of characteristic $p$ (see \cite[Exercise 5.1.11]{Rob}), and every finitely generated torsion-free group can be embedded in $UT(n, \mathbb{Z})$ (see \cite[3.3.4]{LR}). 

In \cite{Rom} the above mentioned algorithm has been applied to solve the PSP in $UT(n,F_p)$. In addition, a ``similar algorithm'' (but without details) has been proposed for the same purpose in $UT(n,\mathbb{Z})$. On the other hand, in \cite[Subsection 2.2]{Mah12}, for a finite field $K$, it has been shown that the PSP for a matrix in $UT(4,K)$ (provided that the entries of the superdiagonal are not all zero!) can be reduced to the PSP in the additive group of $K$. This can be generalized as follows.

\begin{proposition}\label{UT}
Let $R$ be a commutative ring with identity. Then solving the PSP in $UT(n,R)$ is equivalent to solve the PSP in the additive group of $R$.
\end{proposition}

\begin{proof} Let $g\in UT(n,R)$ and suppose that $g=(b_{ij})$ is not the identity matrix. It is enough to prove that for any $m\geq 1$, if $g^a=(c_{ij})$, then there exists $k\geq 1$ such that $0\neq c_{i(i+k)}=ab_{i(i+k)}$. This follows by induction on $m$. In fact, if $b_{i(i+1)}\neq 0$ for some $i$, then $c_{i(i+1)}=ab_{i(i+1)}$; if $b_{i(i+l)}=0$ and $b_{i(i+k)}\neq 0$ for some $k,l$ such that $1\leq l<k\leq n-i $, then $c_{i(i+k)}=ab_{i(i+k)}$. 
\end{proof}

Since the DLP in the additive group of $\mathbb{Z}_m$ or $\mathbb{Z}$ is easy, Proposition \ref{UT} implies that the groups $UT(n,\mathbb{Z}_m)$ and $UT(n,\mathbb{Z})$ are not suitable for Protocols I and II. 

Finally notice that, in contrast with \cite{Rom}, solving the PSP in finite $p$-groups does not guarantee that this is possible for a finite nilpotent group of order $n$, because the decomposition of $n$ in prime factors could be unknown. A similar argument holds for the torsion subgroup of a finitely generated infinite nilpotent group.

\bigskip

\noindent {\bf Acknowledgments}. The last two authors are members of the ``National Group for Algebraic and Geometric Structures, and their Applications'' (GNSAGA -- INdAM) and they would like to thank the Department of Computer Science of the University of York for the excellent hospitality while part of this paper was being written. Their research is supported by a grant of the University of Campania ``Luigi Vanvitelli'', in the framework of Programma V:ALERE 2019.


\end{document}